\documentclass[12pt]{amsart}
\usepackage{amsfonts,amsmath,amssymb,amsthm, geometry}
\usepackage[T1]{fontenc}
\usepackage[latin9]{inputenc}
\usepackage[all]{xy}
\usepackage[active]{srcltx}
\usepackage{graphicx}

\makeatletter

\newtheorem{theorem}{Theorem}[section]
\newtheorem{corollary}[theorem]{Corollary}
\newtheorem{lemma}[theorem]{Lemma}
\newtheorem{proposition}[theorem]{Proposition}
\theoremstyle{definition}
\newtheorem{definition}[theorem]{Definition}
\theoremstyle{remark}
\newtheorem{remark}[theorem]{\sc Remark}
\newtheorem{example}[theorem]{\sc Example}

\newenvironment{lyxlist}[1]
{\begin{list}{}
{\settowidth{\labelwidth}{#1}
 \setlength{\leftmargin}{\labelwidth}
 \addtolength{\leftmargin}{\labelsep}
 }}
{\end{list}}

\renewcommand{\int}{{\mathrm{int}}}

\newcommand{\Sing}{\mathrm{Sing\hspace{1pt}}}

\newcommand{\supp}{{\mathrm{supp}}}

\newcommand{\fin}{\hspace*{\fill}$\Box$}

\newcommand{\bR}{{\mathbb R}}
\newcommand{\bC}{{\mathbb C}}

\begin{document}

\title{ Milnor fibration at infinity for mixed polynomials}

\date{\today}

\author{Ying Chen}
\address{Math\'ematiques, Laboratoire Paul Painlev\'e, Universit\'e Lille 1,
59655 Villeneuve d'Ascq, France.}
\email{Ying.Chen@math.univ-lille1.fr}

\subjclass[2000]{14D06, 58K05, 57R45, 14P10, 32S20, 58K15}

\keywords{fibrations on spheres, bifurcation locus, Newton polyhedron, regularity at infinity, mixed polynomials}

\begin{abstract}
We study the existence of Milnor fibration on a big enough sphere at infinity for a mixed polynomial  $f: \bR^{2n} \to \bR^2$. By using strong non-degeneracy condition, we prove a counterpart of N\'emethi
and Zaharia's fibration theorem. In particular, we obtain a global version of Oka's fibration theorem for strongly non-degenerate and convenient mixed polynomials.
\end{abstract}
\maketitle

\section{Introduction}\label{s:intro}
In the local case of germs of holomorphic polynomial functions with isolated singularities, it is well known that there exists a locally trivial fibration $\varphi:=\frac{f}{\left|f\right|}:S_{r}\setminus f^{-1}(0)\rightarrow S^{1}$
in a sufficiently small sphere which is called Milnor fibration, see \cite{Mi}. Unfortunately, in the global case of holomorphic polynomials, the Milnor fibration $\frac{f}{\left|f\right|}$
at infinity does not exist in general. There are some special cases where $f:\mathbb{C}^{n}\rightarrow\mathbb{C}$ has no atypical values at infinity, for instance:  ``convenient polynomials with non-degenerate Newton principal part at infinity'' (Kouchnirenko \cite{Ku}), polynomials which are ``tame'' (Broughton \cite{Br1}, \cite{Br2}), ``quasi-tame'' (N\'emethi
\cite{Ne1}, \cite{Ne2}). In these cases, the Milnor fibration $\frac{f}{\left|f\right|}$ at infinity exists in a sufficiently large sphere which is equivalent to the fibration $f:f^{-1}(S_{R}^{1})\rightarrow S_{R}^{1}$ for $R$ sufficiently. In \cite{NZ2}, N\'emethi and Zaharia considered a special class of
holomorphic polynomials called ``semitame'' whose atypical values are contained in $\left\{0\right\}$. It was shown that for semitame polynomials, the Milnor fibration $\frac{f}{\left|f\right|}$ at infinity exists. When $n=2$, A.Bodin
in \cite{Bo} proved that the Milnor fibration $\frac{f}{\left|f\right|}$ at infinity exists if and only if $f$ is semitame.
 
Recently, Oka introduced the terminology of ``mixed polynomials'' which is a polynomial function $\mathbb{C}^{n}\rightarrow\mathbb{C}$ with variables $\mathbf{z}$ and $\overline{\mathbf{z}}$, therefore a real polynomial application $\mathbb{R}^{2n}\rightarrow\mathbb{R}^{2}$. By defining non-degeneracy conditions for mixed function
germs, Oka showed in \cite[Theorem 29, 33, 36]{Oka2} that: for a strongly non-degenerate convenient mixed function germ $f:(\mathbb{C}^{n},0)\rightarrow(\mathbb{C},0)$, there exist positive
numbers $r_{0},\delta_{0}$ and $\delta\ll \delta_{0}$, such that for any $r\leq r_{0}$,
$$f:f^{-1}(D_{\delta}^{*})\cap B_{r}^{2n}\rightarrow D_{\delta}^{*}$$
is locally trivial fibrations and the topological isomorphism class does not depend on the choice of $r$ and $\delta_{0}$. Moreover,
$\varphi:=\frac{f}{\left|f\right|}:S_{r}^{2n-1}\setminus K_{r}\longrightarrow S^{1}$ is also a locally trivial fibration which is equivalent to the above fibration.
For mixed polynomials, one can also ask under which condition does the Milnor fibration $\frac{f}{\left|f\right|}$ at infinity exist? In this paper, we get approach to this problem by using the strong non-degeneracy condition at infinity defined in \cite{CT}.
Consider a mixed polynomial $f:\mathbb{C}^{n}\rightarrow\mathbb{C}$. Inspired by Oka's construction in the local case, we prove a N\'emethi and Zaharia type fibration:
\begin{theorem}
\label{thm:M fib}If $f$ is a Newton strongly non-degenerate mixed
polynomial, then $\exists\delta_{0}>0$ and $R_{0}>0$ sufficient
large such that for any $\delta\geq\delta_{0}$ and $R>R_{0}$ \[
\frac{f}{\left|f\right|}:S_{R}^{2n-1}\setminus f^{-1}(D_{\delta})\longrightarrow S^{1}\]
is a locally trivial fibration for $R\geq R_{0}$ and is equivalent
to the global fibration\[
f_{\mid}:f^{-1}(S_{\delta}^{1})\rightarrow S_{\delta}^{1}.\]
\end{theorem}
Unlike in the semitame setting of holomorphic case, we don't have to suppose any other condition for atypical values of $f$ in the above theorem.  
As a consequence of the above theorem, we get the following global version of \cite[Theorem 29, 33, 36]{Oka2}:
\begin{corollary}\label{cor:mixed global}
If $f$ is a Newton strongly non-degenerate convenient mixed polynomial,
then there exists $R_{0}>0$ sufficient large such that for all $R\geq R_{0}$
the Milnor fibration at infinity\[
\frac{f}{\left|f\right|}:S_{R}^{2n-1}\setminus K\longrightarrow S^{1}\]
exists and is equivalent to the global fibration\[
f_{\mid}:f^{-1}(S_{\delta}^{1})\rightarrow S_{\delta}^{1}\]
where $\delta>0$ is sufficient large.
\end{corollary}
In this paper, we will review some basic definitions and properties of mixed polynomials in Section \ref{s:pre}. In order to get an effective estimation of atypical values of $\frac{f}{\left|f\right|}$, we define the $\rho$-regularity for $\frac{f}{\left|f\right|}$ in Section \ref{s:appro}, which allows us to get a type of formulation
like \cite[Theorem 1.1 ]{CT}. The proof of Theorem \ref{thm:M fib} and Corollary \ref{cor:mixed global} will be given in Section \ref{s:fib}. Our example \ref{ex:semitame} shows that the semitame condition is not sufficient to insure
the existence of the Milnor fibration $\frac{f}{\left|f\right|}$ at infinity in the mixed setting.

\section{Preliminaries}\label{s:pre}
\subsection{Mixed singularity and homogeneous polynomials}
Let $f:=(g,h):\mathbb{R}^{2n}\rightarrow\mathbb{R}^{2}$ be a polynomial
application, where $g(x_{1},\ldots,y_{n})$ and $h(x_{1},\ldots,y_{n})$
are real polynomials. By writing $\mathbf{z}=\mathbf{x}+i\mathbf{y}\in\mathbb{C}^{n}$,
where $z_{k}=x_{k}+iy_{k}$ for $k=1,2\ldots n$, we get a polynomial
function $f:\mathbb{C}^{n}\rightarrow\mathbb{C}$ in variables $\mathbf{z}$
and $\overline{\mathbf{z}}$, namely $f(\mathbf{z},\overline{\mathbf{z}}):=g(\frac{\mathbf{z}+\overline{\mathbf{z}}}{2},\frac{\mathbf{z}-\overline{\mathbf{z}}}{2i})+ih(\frac{\mathbf{z}+\overline{\mathbf{z}}}{2},\frac{\mathbf{z}-\overline{\mathbf{z}}}{2i})$,
and reciprocally for a polynomial function $f:\mathbb{C}^{n}\rightarrow\mathbb{C}$
in variables $\mathbf{z}$ and $\overline{\mathbf{z}}$, we can consider
it as a polynomial application $(\mathrm{Re}f,\mathrm{Im}f)$. Then
$f$ is called a \emph{mixed polynomial}, after \cite{Oka2}. We write $f$ as follows:
\begin{equation}
f(\mathbf{z},\overline{\mathbf{z}})=\sum_{\nu,\mu}c_{v,\mu}\mathbf{z}^{\nu}\mathbf{\overline{z}}^{\mu}\label{eq:mixed}
\end{equation}
where $c_{v,\mu}\neq0$, $\mathbf{z}^{\nu}:=z_{1}^{v_{1}}\cdots z_{n}^{v_{n}}$ and $\mathbf{\overline{z}}^{\mu}:=\overline{z}_{1}^{\mu_{1}}\cdots z_{n}^{\mu_{n}}$
for n-tuples $v=(v_{1},\ldots,v_{n})$, $\mu=(\mu_{1},\ldots,\mu_{n})\in\mathbb{N}^{n}$. In the sequel, given a mixed polynomial $f$, we consider $f$ as in the form of equation \eqref{eq:mixed}.
\\
For a mixed polynomial $f$, we shall often use derivation with respect to $\mathbf{z}$ and $\overline{\mathbf{z}}$ such as in the following notations:
\[
\mathrm{d}f:=\left(\frac{\partial f}{\partial z_{1}},\cdots,\frac{\partial f}{\partial z_{n}}\right),\overline{\mathrm{d}}f:=\left(\frac{\partial f}{\partial\overline{z}_{1}},\cdots,\frac{\partial f}{\partial\overline{z}_{n}}\right)
\]
\begin{definition}
\label{def:mixed singularity}We call $w$ a mixed singularity
of $f:\mathbb{C}^{n}\rightarrow\mathbb{C}$, if $w$ is a critical
point of the mapping $f:=(g,h):\mathbb{R}^{2n}\rightarrow\mathbb{R}^{2}$.
\end{definition}
By abuse of notation, we continue to denote the set of mixed singularities for a mixed polynomial $f$ by $\Sing f$.
The next proposition give us a straight way to calculate the locus of mixed singularities.
\begin{proposition}
\label{pro:singularity}\cite[Proposition 1]{Oka2} Let $f:\mathbb{C}^{n}\rightarrow\mathbb{C}$
be a mixed polynomial. Then $w\in\mathbb{C}^{n}$ is a mixed singularity
of $f$ if and only if there exists a complex number $\lambda$ with
$\left|\lambda\right|=1$ such that $\overline{\mathrm{d}f}=\lambda\overline{\mathrm{d}}f$.
\end{proposition}
For mixed polynomials, we have two notions of homogeneous polynomials introduced in \cite{Ci} and \cite{Oka1}.
\begin{definition}
\label{def:radially weighted} A mixed polynomial $f:\mathbb{C}^{n}\rightarrow\mathbb{C}$
is called \emph{radial weighted homogeneous} if there exist $n$
integers $q_{1},\ldots,q_{n}$ with $\gcd(q_{1},\ldots,q_{n})=1$
and a positive integer $m_{r}$ such that $\sum_{j=1}^{n}q_{j}(v_{j}+\mu_{j})=m_{r}$
for every n-tuples $\nu$ and $\mu$. We call $(q_{1},\ldots,q_{n})$
the radial weight of $f$ and $m_{r}$ the radial degree
of $f$. More precisely, $f$ is radial weighted homogeneous of type $(q_{1},\ldots,q_{n};m_{r})$
if and only if it verifies the following equation for all $t\in\mathbb{R}^{*}=\mathbb{R}\setminus\left\{ 0\right\}$:\[
f(t\circ\mathbf{z})=(t^{q_{1}}z_{1},\ldots,t^{q_{n}}z_{n},t^{q_{1}}\overline{z}_{1},\ldots,t^{q_{n}}\overline{z}_{n})=t^{m_{r}}f(\mathbf{z},\overline{\mathbf{z}}).\]
\end{definition}
From Definition \ref{def:radially weighted}, we see that if $f:=(g,h):\mathbb{R}^{2n}\rightarrow\mathbb{R}^{2}$
is a radially weighted homogeneous mixed polynomial, then $g$ and $h$ are real
weighted homogeneous polynomial with the same weights and degrees as $f$.
\begin{definition}
\label{def:polar weighted} A mixed polynomial $f:\mathbb{C}^{n}\rightarrow\mathbb{C}$
is called \emph{polar weighted homogeneous} if there exist $n$
integers $p_{1},\ldots,p_{n}$ with $\gcd(p_{1},\ldots,p_{n})=1$
and a positive integer $m_{p}$ such that $\sum_{j=1}^{n}p_{j}(v_{j}-\mu_{j})=m_{p}$
for every n-tuples $\nu$ and $\mu$. We call $(p_{1},\ldots,p_{n})$
the polar weight of $f$ and $m_{r}$ the polar degree
of $f$. More precisely, $f$ is polar weighted homogeneous of type  $(p_{1},\ldots,p_{n};m_{p})$
if and only if it verifies the following equation for all $\lambda\in S^{1}$:\[
f(\lambda\circ\mathbf{z})=f(\lambda^{p_{1}}z_{1},\ldots,\lambda^{p_{n}}z_{n},\lambda^{-p_{1}}\overline{z}_{1},\ldots,\lambda^{-p_{n}}\overline{z}_{n})=\lambda^{m_{p}}f(\mathbf{z},\overline{\mathbf{z}}).\]
\end{definition}
\begin{example}
\label{ex:radial and polar}Let $f,\, g:\mathbb{C}^{2}\rightarrow\mathbb{C}$,
$f(x,y)=\left|x\right|^{2}+\left|y\right|^{2}$ and $g(x,y)=x^{2}+x^{4}\overline{y}^{2}+y^{2}$.
We see that $f$ is a radial weighted homogeneous polynomial of radial
weight $(1,1)$ and degree $2$, but $f$ is not polar weighted homogeneous.
$g$ is a polar weighted homogeneous polynomial of polar weight $(1,1)$
and degree $2$, but $g$ is not radial weighted homogeneous.
\end{example}

\subsection{Newton non-degeneracy at infinity}\label{ss:newton}
In this section, we review the definitions of Newton polyhedron and non-degeneracy conditions introduced in \cite{CT}. Let $f$ be a mixed polynomial:
\begin{definition}\label{def:The-Newton-polyhedron}
We call $\mathrm{supp}\left(f\right)=\left\{ \nu+\mu\in\mathbb{N}^{n}\mid c_{\nu,\mu}\neq0\right\}$
the \textit{support} of $f$. We say that
$f$ is \emph{convenient} if the intersection of $\mathrm{supp}\left(f\right)$
with each coordinate axis is non-empty. We denote by $\overline{\mathrm{supp}(f)}$ the
convex hull of the set $\mathrm{supp}(f)\setminus\{0\}$.
\noindent
The \emph{Newton polyhedron} of
a mixed polynomial $f$,
denoted by $\Gamma_{0}(f)$, is the convex hull of the set $\left\{ 0\right\} \cup\mathrm{supp}(f)$. The \emph{Newton boundary at infinity}, denoted by $\Gamma^+(f)$, is the union of the faces of the polyhedron
$\Gamma_{0}(f)$ which do not contain the origin. By ``face'' we
mean face of any dimension.
\end{definition}
\begin{definition}\label{def:The-mixed-polynomial}
For any face $\Delta$ of $\overline{\mathrm{supp}(f)}$, we denote
the restriction of $f$ to $\Delta \cap \supp (f)$ by $f_{\Delta} :=  \sum_{\nu +\mu \in\Delta \cap \supp (f)} c_{\nu,\mu}\mathbf{z}^{\nu}\overline{\mathbf{z}}^{\mu}$.
The mixed polynomial $f$ is called
\emph{ non-degenerate} if
$\Sing f_{\Delta}\cap f_\Delta^{-1}(0) \cap \bC^{*n} = \emptyset$, for each face $\Delta$ of $\Gamma^+(f)$. We say that $f$ is \textit{Newton strongly non-degenerate} if   $\Sing f_{\Delta} \cap \bC^{*n} = \emptyset$ for any face $\Delta$ of $\Gamma^+(f)$.
\end{definition}
It is easily seen that these two non-degeneracy conditions are not equivalent, but they coincide in the holomorphic setting. Let us recall the definition of bad faces for mixed polynomials. (See also \cite{NZ1}, \cite{CT}.)
\begin{definition}\label{d:badface}
 A face $\Delta\subseteq\overline{\mathrm{supp}(f)}$
is called \emph{bad} if:
\begin{lyxlist}{00.00.0000}
\item [{(i)}] there exists a hyperplane $H\subset\mathbb{R}^{n}$ with
equation $a_{1}x_{1}+\cdots+a_{n}x_{n}=0$ (where $x_{1},\ldots,x_{n}$
are the coordinates of $\mathbb{R}^{n}$) such that:
\item [{$\qquad\qquad$}] (a) there exist $i$ and $j$ with $a_{i}<0$
and $a_{j}>0$, \medskip{}
\item [{$\qquad\qquad$}] (b) $H\cap\overline{\mathrm{supp}(f)}=\Delta$.
\end{lyxlist}
Let $\mathfrak{B}$ denote the set of bad faces of $\overline{\mathrm{supp}(f)}$.
 A face $\Delta \in\mathfrak{B}$ is called \textit{strictly bad} if it satisfies in addition the following condition and the
set of strictly bad faces of $\overline{\mathrm{supp}(f)}$ will be denoted by $\mathfrak{SB}$:
\smallskip
\begin{lyxlist}{00.00.0000}
\item [{(ii)}] the affine subspace of the same dimension spanned by $\Delta$
contains the origin.
\end{lyxlist}
\end{definition}
Let us review here the notions of Milnor set and asymptotic $\rho$-non-regular values.
\begin{definition}\label{d:milnorset}
The \emph{Milnor set} of a mixed polynomial $f$ is
\[  M(f)=\left\{ \mathbf{z}\in\mathbb{C}^{n}\mid\exists\lambda\in\mathbb{R} \mbox{ and } \mu\in\mathbb{C}^{*}, \mbox{ such that }\lambda\mathbf{z}=\mu\overline{\mathrm{d}f}(\mathbf{z},\overline{\mathbf{z}})+\overline{\mu}\overline{\mathrm{d}}f(\mathbf{z},\overline{\mathbf{z}})\right\}.\]
\end{definition}
\begin{definition}\label{d:s}
The set of \emph{asymptotic} \emph{$\rho$-non-regular values} of
a mixed polynomial $f$ is \begin{align*}
S(f) & =\{c\in\mathbb{C\mid}\exists\,\{\mathbf{z}_{k}\}_{k\in\mathbb{N}}\subset M(f),\,\underset{k\rightarrow\infty}{\lim}\Vert\mathbf{z}_{k}\Vert=\infty\,\mathrm{and}\,\underset{k\rightarrow\infty}{\lim}\mathit{f(\mathbf{z}_{k},\overline{\mathbf{z}_{k}})=c}\}.
\end{align*}
\end{definition}

\section{Approximation of atypical values of $\frac{f}{\left|f\right|}$}\label{s:appro}
Let us denote by $\varphi$ the function $\frac{f}{\left|f\right|}:\mathbb{C}^{n}\setminus V(f)\rightarrow S^{1}$
where $V(f)=f^{-1}(0)$. To simplify notation, we continue to write $\mathrm{d}\varphi$ and $\overline{\mathrm{d}}\varphi$ specifically for the partial derivatives of the variables $\mathbf{z}$ and $\overline{\mathbf{z}}$.
\begin{lemma}
For \textbf{$\mathbf{z}\in\mathbb{C}^{n}\setminus V(f)$}, the fibre
$\varphi^{-1}(\varphi(\mathbf{z},\overline{\mathbf{z}}))$ does not
intersect transversely the sphere  $S_{\left\Vert \mathbf{z}\right\Vert }^{2n-1}$ at \textbf{$\mathbf{z}\in\mathbb{C}^{n}$},
if and only if there exists $\lambda\in\mathbb{R}$, such that
\begin{equation}
\lambda\mathbf{z}=i\overline{f}\,\overline{\mathrm{d}}f(\mathbf{z},\overline{\mathbf{z}})-if\overline{\mathrm{d}f}(\mathbf{z},\overline{\mathbf{z}}).\label{eq:5-1}
\end{equation}
In particular, $\Sing\varphi=\{\mathbf{z}\in\mathbb{C}^{n}\setminus V(f)\mid\overline{f}\,\overline{\mathrm{d}}f(\mathbf{z},\overline{\mathbf{z}})=f\overline{\mathrm{d}f}(\mathbf{z},\overline{\mathbf{z}})\}$.
\end{lemma}
\begin{proof}
Observe first $\varphi=-\mathrm{Re}(i\log f)$. By \cite[Lemma 2.1]{CT},
the non-transversality of the fiber $\varphi^{-1}(\varphi(\mathbf{z},\overline{\mathbf{z}}))$ and the sphere  $S_{\left\Vert \mathbf{z}\right\Vert }^{2n-1}$
implies:
\begin{equation}
\gamma\mathbf{z}=\mu\overline{\mathrm{d}\varphi}(\mathbf{z},\overline{\mathbf{z}})+\overline{\mu}\,\overline{\mathrm{d}}\varphi(\mathbf{z},\overline{\mathbf{z}}),\label{eq:5-2}
\end{equation}
for some $\gamma\in\mathbb{R}$ and $\mu\in\mathbb{R}^{*}$. By definition
of $\overline{\mathrm{d}\varphi}$ and $\overline{\mathrm{d}}\varphi$,
we have:
\begin{alignat*}{1}
\overline{\mathrm{d}}\varphi(\mathbf{z},\overline{\mathbf{z}}) & =-\overline{\mathrm{d}}\mathrm{Re}(i\log f)=i\frac{\overline{\mathrm{d}}f(\mathbf{z},\overline{\mathbf{z}})}{\overline{f}}\\
\overline{\mathrm{d}\varphi}(\mathbf{z},\overline{\mathbf{z}}) & =-\overline{\mathrm{d\mathrm{Re}(i\log f)}}=-i\frac{\overline{\mathrm{d}f}(\mathbf{z},\overline{\mathbf{z}})}{f}.
\end{alignat*}
Multipling the two sides of \eqref{eq:5-2} by $\left|f\right|^{2}$,
we conclude \eqref{eq:5-1}, where $\lambda=\frac{\gamma}{\mu}\left|f\right|^{2}\in\mathbb{R}$.
In particular, taking $\lambda=0$ in \eqref{eq:5-1}, we obtain $\Sing\varphi$.
\end{proof}
Combining with the above lemma, we are led to define $\rho$-regularity for $\varphi$. (In general case, this regularity condition is defined in \cite{ACT})
\begin{definition}\label{d:phi}
We call \emph{$\rho$-non-regular locus} of $\varphi$ the semi-algebraic set:\[
M(\varphi)=\left\{ \mathbf{z}\in\mathbb{C}^{n}\setminus V(f)\mid\exists\lambda\in\mathbb{R},\,\mathrm{such\, that}\,\lambda\mathbf{z}=i\bar{f}\overline{\mathrm{d}}f(\mathbf{z},\overline{\mathbf{z}})-if\overline{\mathrm{d}f}(\mathbf{z},\overline{\mathbf{z}})\right\} .\]
and we call \emph{asymptotic $\rho$-non-regular values} of
$\frac{f}{\left|f\right|}$ the set:\[
S(\varphi)  =\{c\in S^{1}\mathbb{\mid}\exists\,\{\mathbf{z}_{k}\}_{k\in\mathbb{N}}\subset M(\varphi),\,\underset{k\rightarrow\infty}{\lim}\Vert\mathbf{z}_{k}\Vert=\infty\,\\
  \,\mathrm{and}\:\mathrm{}\underset{k\rightarrow\infty}{\lim}\mathit{\varphi(\mathbf{z}_{k},\overline{\mathbf{z}_{k}})=c}\}.\]
\end{definition}
Recall the notations of Milnor set $M(f)$ and the asymptotic $\rho$-non-regular set $S(f)$.
The above definition enables us to obtain the following structure result of $S(\varphi)$.
\begin{lemma}
\label{lem:sag}$S(\varphi)$ is semi-algebraic and $M(\varphi)\subset M(f)\setminus V(f)$.
\end{lemma}
\begin{proof}
The inclusion of $M(\varphi)\subset M(f)\setminus V(f)$ follows from the definitions of \ref{d:milnorset} and \ref{d:phi}.
Since $M(\varphi)$ is a semi-algebraic set, we now proceed analogously to the proof of \cite[Proposition 2.1]{CT}
and we see that $S(\varphi)$ is semi-algebraic.
\end{proof}
Our next proposition shows that under some homogeneous condition, $\Sing\varphi$ could be equal to $M(\varphi)$.
\begin{proposition}\label{pro:homogeneous}
If $f$ is a mixed radial weighted homogeneous polynomial and not constant, then $\Sing\varphi=\Sing f\setminus V(f)=M(\varphi)$.
\end{proposition}
\begin{proof}
Let us denote the radial weights of $f$ by $q_{1},\cdots,q_{n}$ and the radial degree of $f$ by $m_{r}$, where $q_{1},\cdots,q_{n}\in\mathbb{Z}$
and $m_{r}\neq0$. First, we have $\Sing\varphi\subset\Sing f\setminus V(f)$ and $\Sing\varphi\subset M(\varphi)$. To prove the equality, let $\mathbf{a}\in\Sing f$ and $f(\mathbf{a},\overline{\mathbf{a}})\neq0$.
Therefore $\exists\lambda\in S^{1}$ such that for $1\leq i\leq n$:
\begin{equation}
\overline{\frac{\partial f}{\partial z_{i}}}(\mathbf{a},\overline{\mathbf{a}})=\lambda\frac{\partial f}{\partial\overline{z}_{i}}(\mathbf{a},\overline{\mathbf{a}}).\label{eq:5-3}
\end{equation}
Since $f$ is radial weighted homogeneous, by Euler's lemma, we have:
\begin{equation}
\sum_{i=1}^{n}q_{i}a_{i}\frac{\partial f}{\partial z_{i}}(\mathbf{a},\overline{\mathbf{a}})+\sum_{i=1}^{n}q_{i}\overline{a}_{i}\frac{\partial f}{\partial\overline{z}_{i}}(\mathbf{a},\overline{\mathbf{a}})=m_{r}f(\mathbf{a},\overline{\mathbf{a}}).\label{eq:5-4}
\end{equation}
Let $A=\sum_{i=1}^{n}q_{i}a_{i}\frac{\partial f}{\partial z_{i}}(\mathbf{a},\overline{\mathbf{a}})$
and $B=\sum_{i=1}^{n}q_{i}\overline{a}_{i}\frac{\partial f}{\partial\overline{z}_{i}}(\mathbf{a},\overline{\mathbf{a}})$.
Multiplying \eqref{eq:5-3} by $q_{i}\overline{a}_{i}$, we obtain:
\begin{equation}
\overline{A}=\lambda B\label{eq:5-5}
\end{equation}
which implies $A\overline{A}=B\overline{B}$ since $\lambda\in S^{1}$.
From \eqref{eq:5-4}, \eqref{eq:5-5} and $f(\mathbf{a},\overline{\mathbf{a}})\neq0$, we therefore get $AB\neq0$.
Consequently,
\begin{alignat}{1}
\frac{\overline{f(\mathbf{a},\overline{\mathbf{a}})}}{f(\mathbf{a},\overline{\mathbf{a}})} & =\frac{\overline{A}+\overline{B}}{A+B}=\frac{B\overline{A}+B\overline{B}}{B(A+B)}=\frac{B\overline{A}+A\overline{A}}{B(A+B)}=\lambda
\end{alignat}
which proves that $\mathbf{a}\in\Sing\varphi$ from \ref{eq:5-3}.
Thus, we have $\Sing\varphi=\Sing f\setminus V(f)$. Using Euler vector field as in the proof of \cite[Proposition 3.2]{ACT},
we have $M(\varphi)\subset\Sing f\setminus V(f)$. This finishes the proof.
\end{proof}
For simplicity of notation, we write $\varphi_{\triangle}:=\frac{f_{\triangle}}{\left|f_{\triangle}\right|}$ for the restriction of $\frac{f}{\left|f\right|}$,
where $\triangle$ is a face of $\overline{\mathrm{supp}(f)}$.
\begin{theorem}\label{thm:appro}
If $f$ is Newton strongly non-degenerate at infinity for any face of $\overline{\mathrm{supp}(f)}$, then $M(\varphi)$ is bounded and $S(\varphi)=\emptyset.$
\end{theorem}
\begin{proof}
Assume that $M(\varphi)$ is not bounded, then by Curve selection lemma at infinity (\cite[Lemma 2]{NZ2}, \cite[Lemma 2.3]{CT}),
there exists $\mathbf{z}(t)$ of $M(\varphi)$ a real analytic path
defined on a small enough interval $\left]0,\varepsilon\right[$ such
that  \[
\underset{t\rightarrow0}{\lim\,}\Vert\mathbf{z}(t)\Vert=\infty.\]
Since $\mathbf{z}(t)\subset M(\varphi)$, there exists a real analytic
curve $\lambda(t)$, such that for $t\in\left]0,\varepsilon\right[$
we have:
\begin{equation}
\lambda(t)\mathbf{z}(t)=i\bar{f}\,\overline{\mathrm{d}}f(\mathbf{z}(t),\overline{\mathbf{z}}(t))-if\overline{\mathrm{d}f}(\mathbf{z}(t),\overline{\mathbf{z}}(t)).\label{eq:5-6}
\end{equation}
Suppose here $\lambda(t)\not\equiv0$ and let $I=\left\{ i\mid z_{i}(t)\not\not\equiv0\right\}$.
Then $I\neq\emptyset$ since $\underset{t\rightarrow0}{\lim}\Vert\mathbf{z}(t)\Vert=\infty$.
Assuming that $I=\left\{ 1,\ldots,m\right\} $, we write the expansions
of $f(\mathbf{z}(t),\overline{\mathbf{z}}(t))$, $\mathbf{z}(t)$
and $\lambda(t)$ explicitly as follows:
\begin{alignat}{1}
z_{i}(t) & =a_{i}t^{p_{i}}+\mathrm{h.o.t.,\qquad\mathit{\mathrm{where\,}a_{i}\neq}0,\,}p_{i}\in\mathbb{Z},\,1\leq i\leq m.\label{eq:5-7}\\
f(\mathbf{z}(t),\overline{\mathbf{z}}(t)) & =
\begin{cases}
bt^{\delta}+\mathrm{h.o.t.},\qquad\mathrm{where}\, b\in\mathbb{C}^{*},\,\delta\neq0, & \mathrm{if}\,\underset{t\rightarrow0}{\lim\,}f(\mathbf{z}(t),\overline{\mathbf{z}}(t))=0\,\mathrm{or}\,\infty\\
c+bt^{\delta}+\mathrm{h.o.t.}\qquad\mathrm{where}\, c,b\in\mathbb{C}^{*},\,\delta\neq0, & \mathrm{if}\,\underset{t\rightarrow0}{\lim\,}f(\mathbf{z}(t),\overline{\mathbf{z}}(t))=c.
\end{cases}\label{eq:5-8}\\
\lambda(t) & =\lambda_{0}t^{\gamma}+\mathrm{h.o.t.,\qquad where\,\lambda_{0}\in\mathbb{R}}^{*},\,\gamma\in\mathbb{Z},\,\lambda(t)\in\mathbb{R}.\label{eq:5-9}
\end{alignat}
Set $\mathbf{a}=(a_{1},\ldots,a_{m})\in\mathbb{C}^{*I}$, $\mathbf{P}=(p_{1},\ldots,p_{m})\in\mathbb{R}^{m}$
and consider the linear function $l_{\mathbf{P}}=\sum_{i=1}^{m}p_{i}x_{i}$
defined on $\overline{\mathrm{supp(\mathit{f}^{I})}}$. Let $\triangle$
be the \emph{maximal face} of $\overline{\mathrm{supp(\mathit{f}^{I})}}$
where $l_{\mathbf{P}}$ takes its minimal value, say this value is
$d_{\mathbf{P}}$. We have:
\begin{equation}
f(\mathbf{z}(t),\overline{\mathbf{z}}(t))=f_{\triangle}^{I}(\mathbf{a},\overline{\mathbf{a}})t^{d_{\mathbf{P}}}+\mathrm{h.o.t.}\label{eq:5-10}
\end{equation}
Let us discuss the following two cases:

(I). If $\underset{t\rightarrow0}{\lim}\, f(\mathbf{z}(t),\overline{\mathbf{z}}(t))=0$
or $\infty$, we get $d_{\mathbf{P}}\leq\delta$. Since $\underset{t\rightarrow0}{\lim}\Vert\mathbf{z}(t)\Vert=\infty$,
this implies $p:=\underset{j\in I}{\min}\{p_{j}\}<0$. Now using \eqref{eq:5-7}-\eqref{eq:5-10} in \eqref{eq:5-6},
we get:
\begin{equation}
i\overline{b}\frac{\partial f_{\triangle}^{I}}{\partial\overline{z}_{i}}(\mathbf{a},\overline{\mathbf{a}})-ib\frac{\overline{\partial f_{\triangle}^{I}}}{\partial z_{i}}(\mathbf{a},\overline{\mathbf{a}})=
\begin{cases}
\lambda_{0}a_{i}, & \quad\mathrm{if}\, d_{\mathbf{P}}-p_{i}+\delta=p_{i}+\gamma.\\
\\0, & \quad\mathbf{\mathrm{if}}\, d_{\mathbf{P}}-p_{i}+\delta<p_{i}+\gamma.
\end{cases}\label{eq:5-11}
\end{equation}
Let $J=\left\{ j\mid d_{\mathbf{P}}-p_{j}+\delta=p_{j}+\gamma\right\}$.
We suppose $J\neq\emptyset$ which gives $J=\{j\mid p_{j}=p=\underset{1\leq j\leq m}{\min}\{p_{j}\}<0\}$.
Consider the derivative of $f(\mathbf{z}(t),\overline{\mathbf{z}}(t))$ with respect to
$t$. On one hand, we have:
\begin{equation}
\frac{\mathrm{d}f(\mathbf{z}(t),\overline{\mathbf{z}}(t))}{\mathrm{d}t}=b\delta t^{\delta-1}+\mathrm{h.o.t}.\label{eq:5-12}
\end{equation}
On the other hand, we have:
\begin{alignat}{1}
\frac{\mathrm{d}f(\mathbf{z}(t),\overline{\mathbf{z}}(t))}{\mathrm{d}t} & ={\displaystyle \sum_{i=1}^{m}(\frac{\partial f}{\partial z_{i}}\cdot\frac{\partial z_{i}}{\partial t}+\frac{\partial f}{\partial\overline{z_{i}}}\cdot\frac{\partial\overline{z_{i}}}{\partial t})}\label{eq:5-13}\\
 & =\left[\left\langle \mathbf{Pa},\overline{\mathrm{d}f_{\triangle}^{I}}(\mathbf{a},\overline{\mathbf{a}})\right\rangle +\left\langle \mathbf{P}\mathbf{\overline{\mathbf{a}}},\overline{\bar{\mathrm{d}}f_{\triangle}^{I}(\mathbf{a},\overline{\mathbf{a}})}\right\rangle \right]t^{d_{\mathbf{P}}-1}+\mathrm{h.o.t}.\nonumber
\end{alignat}
where $\mathbf{Pa}=(p_{1}a_{1},\ldots,p_{m}a_{m})$. From \eqref{eq:5-11},
we obtain:
\begin{equation}
\mathrm{Re}\left\langle \mathbf{P}\mathbf{a},i\overline{b}\bar{\mathrm{d}}f_{\triangle}^{I}(\mathbf{a},\overline{\mathbf{a}})-ib\overline{\mathrm{d}f_{\triangle}^{I}}(\mathbf{a},\overline{\mathbf{a}})\right\rangle =\sum_{i\in J}\lambda_{0}p\Vert a_{j}\Vert^{2}\neq0.\label{eq:5-14}
\end{equation}
If $d_{\mathbf{P}}<\delta$, then comparing the orders of the expansions
\eqref{eq:5-12} and \eqref{eq:5-13} with respect to $t$, we have $\left\langle \mathbf{Pa},\overline{\mathrm{d}f_{\triangle}^{I}}(\mathbf{a},\overline{\mathbf{a}})\right\rangle +\left\langle \mathbf{P}\mathbf{\overline{\mathbf{a}}},\overline{\bar{\mathrm{d}}f_{\triangle}^{I}(\mathbf{a},\overline{\mathbf{a}})}\right\rangle =0$
and $f_{\triangle}^{I}(\mathbf{a},\overline{\mathbf{a}})=0$. Multiplying \eqref{eq:5-13} by $i\overline{b}$ and comparing the real parts of the equality, we obtain a contradiction with \eqref{eq:5-14}.
If $d_{\mathbf{P}}=\delta$, then by \eqref{eq:5-12}, we have $\mathrm{Re}\left\langle \mathbf{P}\mathbf{a},i\overline{b}\bar{\mathrm{d}}f_{\triangle}^{I}(\mathbf{a},\overline{\mathbf{a}})-ib\overline{\mathrm{d}f_{\triangle}^{I}}(\mathbf{a},\overline{\mathbf{a}})\right\rangle =\mathrm{Re}(i\left|b\right|^{2}\delta)=0$,
which contradicts \eqref{eq:5-14}. It follows that
$J=\emptyset$. Hence $\mathbf{a}\in\mathbb{C}^{*I}$ is a singularity
of $f_{\triangle}^{I}$ and $f_{\triangle}^{I}(\mathbf{a},\overline{\mathbf{a}})=b$.
By \cite[Remark 3.3]{CT}, this is contrary to the strong non-degeneracy of $f^{I}$.

(II). If $\underset{t\rightarrow0}{\lim}\, f(\mathbf{z}(t),\overline{\mathbf{z}}(t))=c\in\mathbb{C}^{*}$,
comparing the orders of the expansions \eqref{eq:5-9} and \eqref{eq:5-11} with respect to $t$, we have $d_{\mathbf{P}}<\delta$.
Now using \eqref{eq:5-8}-\eqref{eq:5-11} in \eqref{eq:5-7}, we get:
\begin{equation}
i\overline{c}\frac{\partial f_{\triangle}^{I}}{\partial\overline{z}_{i}}(\mathbf{a},\overline{\mathbf{a}})-ic\frac{\overline{\partial f_{\triangle}^{I}}}{\partial z_{i}}(\mathbf{a},\overline{\mathbf{a}})=
\begin{cases}
\lambda_{0}a_{i}, & \quad\mathrm{if}\, d_{\mathbf{P}}-p_{i}=p_{i}+\gamma.\\
\\0, & \quad\mathbf{\mathrm{if}}\, d_{\mathbf{P}}-p_{i}<p_{i}+\gamma.
\end{cases}\label{eq:5-15}
\end{equation}
Let $J=\left\{ j\mid d_{\mathbf{P}}-p_{j}=p_{j}+\gamma\right\} $.
We suppose $J\neq\emptyset$ which implies $J=\{j\mid p_{j}=p=\underset{1\leq j\leq m}{\min}\{p_{j}\}<0\}$.
We derive $f(\mathbf{z}(t),\overline{\mathbf{z}}(t))$ with respect to $t$. On one hand, we get \eqref{eq:5-12}. On the other hand, we
have \eqref{eq:5-13}. From \eqref{eq:5-15}, we obtain:
\begin{equation}
\mathrm{Re}\left\langle \mathbf{P}\mathbf{a},i\overline{c}\bar{\mathrm{d}}f_{\triangle}^{I}(\mathbf{a},\overline{\mathbf{a}})-ic\overline{\mathrm{d}f_{\triangle}^{I}}(\mathbf{a},\overline{\mathbf{a}})\right\rangle =\sum_{i\in J}\lambda_{0}p\Vert a_{j}\Vert^{2}\neq0.\label{eq:5-16}
\end{equation}
Since $d_{\mathbf{P}}<\delta$, comparing the orders of with respect to $t$ in \eqref{eq:5-12} and \eqref{eq:5-13}, we have $\left\langle \mathbf{Pa},\overline{\mathrm{d}f_{\triangle}^{I}}(\mathbf{a},\overline{\mathbf{a}})\right\rangle +\left\langle \mathbf{P}\mathbf{\overline{\mathbf{a}}},\overline{\bar{\mathrm{d}}f_{\triangle}^{I}(\mathbf{a},\overline{\mathbf{a}})}\right\rangle =0$.
Multiplying \eqref{eq:5-13} by $i\overline{c}$ and comparing the real parts, we obtain a contradiction with \eqref{eq:5-16}.
It follows that $J=\emptyset$. Hence $\mathbf{a}\in\mathbb{C}^{*I}$ is a singularity of $f_{\triangle}^{I}$ and $f_{\triangle}^{I}(\mathbf{a},\overline{\mathbf{a}})=0$.
By \cite[Remark 3.3]{CT}, this is contrary to the non-degeneracy of $f^{I}$.

In general, if we do not assume the strong non-degeneracy of $f$ and let $\mathbf{A}=(\mathbf{a},1,1,\ldots,1)$ with the $i^{th}$ coordinate $z_{i}=1$ for $i\notin I$, then we
have the following conclusion:
\begin{enumerate}
 \item  If $d_{\mathbf{P}}<\delta$, then $\mathbf{A}$ is a singularity of $V(f_{\triangle})$.
 \item  If $d_{\mathbf{P}}=\delta$, then $\mathbf{A}\in\Sing\varphi=\Sing f_{\triangle}\setminus V(f_{\triangle})$ by Proposition \ref{pro:homogeneous}.
\end{enumerate}
When $\lambda(t)\equiv0$, by comparing the orders with respect to $t$ in \eqref{eq:5-6}, we have:
\begin{equation}
\begin{cases}
i\overline{b}\frac{\partial f_{\triangle}^{I}}{\partial\overline{z}_{i}}(\mathbf{a},\overline{\mathbf{a}})-ib\frac{\overline{\partial f_{\triangle}^{I}}}{\partial z_{i}}(\mathbf{a},\overline{\mathbf{a}})=0, & \mathrm{if}\,\underset{t\rightarrow0}{\lim}f(\mathbf{z}(t),\overline{\mathbf{z}}(t))=0\,\mathrm{or}\,\infty.\\
i\overline{c}\frac{\partial f_{\triangle}^{I}}{\partial\overline{z}_{i}}(\mathbf{a},\overline{\mathbf{a}})-ic\frac{\overline{\partial f_{\triangle}^{I}}}{\partial z_{i}}(\mathbf{a},\overline{\mathbf{a}})=0, & \mathrm{if}\,\underset{t\rightarrow0}{\lim}f(\mathbf{z}(t),\overline{\mathbf{z}}(t))=c\in\mathbb{C}^{*}.
\end{cases}
\label{eq:5-18}
\end{equation}
It follows that $\mathbf{a}\in\mathbb{C}^{*I}$ is a singularity of $f_{\triangle}^{I}$. By \cite[Remark 3.3]{CT}, this is contrary
to the non-degeneracy of $f^{I}$. Hence $M(\varphi)$ is bounded and $S(\varphi)=\emptyset$.
\end{proof}
We now proceed to formulate the analogue of \cite[Theorem 1.1]{CT}.
Recall the notation $\mathfrak{SB}$ the union of strictly bad faces of $\overline{\mathrm{supp}(f)}$.
\begin{theorem}\label{thm:appro1}
Let $f:\mathbb{C}^{n}\rightarrow\mathbb{C}$ be a mixed polynomial. Suppose that $f$ is Newton strongly non-degenerate
polynomial which depends effectively on all the variables. Let $f(0)=0$
and $0\notin S(f)$. Then:
\[
S(\varphi)\subset\underset{\triangle\in\mathfrak{SB}}{\bigcup}\varphi_{\triangle}(\Sing\varphi_{\triangle}\bigcap\mathbb{C}^{*n}).
\]
\end{theorem}
\begin{proof}
We use the same notations as in the proof of Theorem \ref{thm:appro}.
For any $c\in S(\varphi)$, by Curve selection Lemma at infinity, there
exists $\mathbf{z}(t)$ of $M(\varphi)$ a real analytic path defined
on a small enough interval $\left]0,\varepsilon\right[$ such that
\[
\underset{t\rightarrow0}{\lim}\Vert\mathbf{z}(t)\Vert=\infty,\,\mathrm{and\,}\underset{t\rightarrow0}{\lim}\varphi(\mathbf{z}(t),\overline{\mathbf{z}}(t))=c_{0}\]
where either $f(\mathbf{z}(t),\overline{\mathbf{z}}(t))=bt^{\delta}+\mathrm{h.o.t}.$
and $d_{\mathbf{P}}\leq\delta<0,c_{0}=\frac{b}{\left|b\right|}$,
or $f(\mathbf{z}(t),\overline{\mathbf{z}}(t))=c+bt^{\delta}+\mathrm{h.o.t.}$
and $d_{\mathbf{P}}\leq0,c\in\mathbb{C}^{*},c_{0}=\frac{c}{\left|c\right|}$.
Consider $\lambda(t)\not\equiv0$, if $\mathrm{ord_{t}}(f(\mathbf{z}(t),\overline{\mathbf{z}}(t)))<0$,
we are in case (I) as in the proof of Theorem \ref{thm:appro},
then we get that $\mathbf{a}\in\mathbb{C}^{*I}$ is a singularity
of $f_{\triangle}^{I}$. If $\mathrm{ord_{t}}(f(\mathbf{z}(t),\overline{\mathbf{z}}(t)))=0$,
we are in case (II) as in the proof of Theorem \ref{thm:appro},
then $\mathbf{a}\in\mathbb{C}^{*I}$ is a singularity of $f_{\triangle}^{I}$.

Set $A=(\mathbf{a},1,1,\ldots,1)$ with the $i^{th}$ coordinate $z_{i}=1$
for $i\notin I$. By recalling the definition of Newton boundary at infinity for mixed polynomial, we have the following two cases:

(I). If $d_{\mathbf{P}}<0$, then, from \cite[Lemma 3.1]{CT}, we conclude that $\triangle$
is a face of $\Gamma^{+}(f^{I})$. On the other hand since $\mathbf{a}$
is a singularity of $f_{\triangle}^{I}$, by Remark  this
contradicts the Newton strong non degeneracy of $f^{I}$.

(II). If $d_{\mathbf{P}}=\delta=0$, then, from \cite[Lemma 3.1]{CT},
it follows that either $\triangle$ is a face of $\Gamma^{+}(f^{I})$ or $\triangle$ satisfies
condition (ii) of Definition \ref{d:badface}. Assume first $\triangle$ is
a face of $\Gamma^{+}(f^{I})$, then we get the same contradiction as
that in (I). Thus $\triangle$ verifies condition (ii) of Definition \ref{d:badface}.
We proceed to show that $\triangle$ is strictly bad face of $\overline{\mathrm{supp(\mathit{f})}}$.
Let us denote by $d$ the minimal value of the restriction of $l_{\mathbf{P}}$
to $\overline{\mathrm{supp}(f)}$. Since $\overline{\mathrm{supp(\mathit{f^{I}})}}=\overline{\mathrm{supp}(f)}\cap\mathbb{R}_{+}^{I}$,
we have $d\leq d_{\mathbf{P}}=0$. Let $H$ be the hyperplane of the
equation $\sum_{i=1}^{m}p_{i}x_{i}+q\sum_{i=m+1}^{n}x_{i}=0$, where
$q>-d+1>0$. Hence for any $\mathbf{x}=(x_{1},\ldots,x_{n})\in\overline{\mathrm{supp}(f)}\setminus\overline{\mathrm{supp}(f^{I})}$,
the value of $\sum_{i=1}^{m}p_{i}x_{i}+q\sum_{i=m+1}^{n}x_{i}$ is
positive. We therefore get $\triangle=\overline{\mathrm{supp(\mathit{f^{I}})}}\cap H=\overline{\mathrm{supp}(f)}\cap H$.
On the other hand, note that $p_{1}=p=\underset{1\leq i\leq m}{\min}\{p_{i}\}<0$
and $q>0$. If $\triangle$ does not satisfy condition (i)(a) of
Definition \ref{d:badface}, then we have $m=n$ and $p_{i}\leq0$ for all $1\leq i\leq n$.
It follows that $f$ can not depend on $z_{1}$ otherwise $d_{\mathbf{P}}$
will be negative. This contradicts the effectiveness
of $f$. Hence we conclude that $\triangle$ is a strictly bad face of
$\overline{\mathrm{supp(\mathit{f})}}$. Since $d_{\mathbf{P}}=0$,
we obtain $c=f_{\triangle}^{I}(\mathbf{a},\overline{\mathbf{a}})=f_{\triangle}(\mathbf{A},\overline{\mathbf{A}})\neq0$.
By $\mathbf{A}\in\Sing\varphi_{\triangle}$ and Proposition \ref{pro:homogeneous},
we get $c_{0}\in\varphi_{\triangle}(\Sing\varphi_{\triangle})$.
When $\lambda(t)\equiv0$, it follows that $\mathbf{a\in\mathbb{C}}^{*I}$
is a singularity of $f_{\triangle}^{I}$ from \eqref{eq:5-18}. In the same manner as above reasoning,
we get the desired conclusion.
\end{proof}
\begin{remark}
\label{rem:conve}
In particular, if a mixed polynomial $f$ is Newton strongly
non-degenerate at infinity and convenient, then by \cite[Corollary 4.1]{CT},
we have $S(f)=\emptyset$. Combining this conclusion with the above theorem, we get $S(\varphi)=\emptyset$
since $\mathfrak{SB}=\emptyset$.
\end{remark}

\section{Fibration at infinity}\label{s:fib}
Recall that for a strongly non-degenerate polynomial $f$, we have the monodromy fibration:
\[
f_{\mid}:f^{-1}(S_{\delta}^{1})\rightarrow S_{\delta}^{1}.\]
over some circle $S_{\delta}^{1}$ of radius $\delta$ which is sufficiently large.
We define two vectors on $\mathbb{C}^{n}\setminus V(f)$:
\begin{eqnarray*}
v_{1}(\mathbf{z},\overline{\mathbf{z}}) & = & \overline{\mathrm{d}\log f}(\mathbf{z},\overline{\mathbf{z}})+\overline{\mathrm{d}}\log f(\mathbf{z},\overline{\mathbf{z}})\\
v_{2}(\mathbf{z},\overline{\mathbf{z}}) & = & i(\overline{\mathrm{d}\log f}(\mathbf{z},\overline{\mathbf{z}})-\overline{\mathrm{d}}\log f(\mathbf{z},\overline{\mathbf{z}})).
\end{eqnarray*}
which have the following geometrical meanings: $v_{1}(\mathbf{z},\overline{\mathbf{z}})$
is the normal vector of $\log\left|f\right|$ and $v_{2}(\mathbf{z},\overline{\mathbf{z}})$
is the normal vector of $-i\log\frac{f}{\left|f\right|}$. In order
to prove Theorem \ref{thm:M fib}, we shall first prove the following proposition.
\begin{proposition}
\label{pro:Under-the-same}Under the same assumption as in Theorem
\ref{thm:M fib}, there exists $\delta_{2}>0$ sufficient large,
such that for any $\mathbf{z}$ of $\{\mathbf{z}\in\mathbb{C}^{n}\mid\left|f(\mathbf{z},\overline{\mathbf{z}})\right|\geq\delta_{2}\}$
the there vectors\[
\mathbf{z},\quad v_{1}(\mathbf{z},\overline{\mathbf{z}}),\quad v_{2}(\mathbf{z},\overline{\mathbf{z}})\]
are either linearly independent over $\mathbb{R}$ or they are linearly
dependent over $\mathbb{R}$ with the following relation \[
\mathbf{z}=av_{1}(\mathbf{z},\overline{\mathbf{z}})+bv_{2}(\mathbf{z},\overline{\mathbf{z}})\]
where $a>0$.
\end{proposition}
\begin{proof}
Since $f$ is strongly non-degenerate at infinity, by \cite[Theorem 1.1]{CT},
$f(\Sing f)\cup S(f)$ is bounded. Let us suppose that $f(\Sing f)\cup S(f)\subset D_{\delta_{1}}$.
For $\left|f(\mathbf{z},\overline{\mathbf{z}})\right|$ sufficient
large we shall prove either $\mathbf{z},v_{1}(\mathbf{z},\overline{\mathbf{z}}),v_{2}(\mathbf{z},\overline{\mathbf{z}})$
are linearly independent over $\mathbb{R}$ or $\mathbf{z}=av_{1}(\mathbf{z},\overline{\mathbf{z}})+bv_{2}(\mathbf{z},\overline{\mathbf{z}})$
where $ab\neq0$. Assume that $\mathbf{z}$ and $v_{2}(\mathbf{z},\overline{\mathbf{z}})$
are linearly dependent over $\mathbb{R}$. By Curve selection Lemma at infinity,
there exist two analytic paths $\mathbf{z}(t)\subset\mathbb{C}^{n}$
and $\lambda(t)\subset\mathbb{R}$ defined on a small enough interval
$\left]0,\varepsilon\right[$ such that
\begin{align}
\underset{t\rightarrow0}{\lim}\,\Vert\mathbf{z}(t)\Vert & =\infty,\,\underset{t\rightarrow0}{\lim}\, f(\mathbf{z}(t),\overline{\mathbf{z}}(t))=\infty.\label{eq:5-19}\\
i(\overline{\mathrm{d}\log f}-\overline{\mathrm{d}}\log f) & (\mathbf{z}(t),\overline{\mathbf{z}}(t))=\lambda(t)\mathbf{z}(t).\label{eq:5-20}
\end{align}
By Lemma \ref{lem:sag}, we have $\mathbf{z}(t)\subset M(f)$.
Thus $\underset{t\rightarrow0}{\lim}\, f(\mathbf{z}(t),\overline{\mathbf{z}}(t))=\infty$
contradicts our condition $f(\Sing f)\cup S(f)\subset D_{\delta_{1}}$.
For $\left|f(\mathbf{z},\overline{\mathbf{z}})\right|$ sufficiently
large, we have actually proved that $\mathbf{z}$ and $v_{2}(\mathbf{z},\overline{\mathbf{z}})$
are linearly independent over $\mathbb{R}$. Since $\overline{\mathrm{d}\log f}+\overline{\mathrm{d}}\log f=\frac{1}{\left|f\right|^{2}}(f\overline{\mathrm{d}f}+\overline{f}\overline{\mathrm{d}}f)$,
a slightly change in the proof of the above linearly independence
shows that for $\left|f(\mathbf{z},\overline{\mathbf{z}})\right|$
sufficient large, $\mathbf{z}$ and $v_{1}(\mathbf{z},\overline{\mathbf{z}})$
are also linearly independent over $\mathbb{R}$. If $v_{1}(\mathbf{z},\overline{\mathbf{z}}),v_{2}(\mathbf{z},\overline{\mathbf{z}})$
are linearly dependent over $\mathbb{R}$, we have $\mathbf{z}\in\Sing f\setminus V(f)$.
Hence $f(\mathbf{z}(t),\overline{\mathbf{z}}(t))\subset f(\Sing f)$.
This contradicts the boundness of $f(\Sing f)$. It follows that
$v_{1}(\mathbf{z},\overline{\mathbf{z}}),v_{2}(\mathbf{z},\overline{\mathbf{z}})$
are linearly independent over $\mathbb{R}$ for $\left|f(\mathbf{z},\overline{\mathbf{z}})\right|$ sufficient large. We are reduce to proving
the proposition for $a>0$. In the remainder of the proof, we assume
$a<0$. By Curve selection Lemma at infinity, there exist the analytic
curves $\mathbf{z}(t)\in\mathbb{C}^{n}$, $a(t)<0$ and $b(t)\in\mathbb{R}$
defined on a small enough interval $\left]0,\varepsilon\right[$ such
that
\begin{align}
\underset{t\rightarrow0}{\lim}\Vert\mathbf{z}(t)\Vert & =\infty,\,\underset{t\rightarrow0}{\lim}f(\mathbf{z}(t),\overline{\mathbf{z}}(t))=\infty.\label{eq:5-21}\\
\mathbf{z}(t) & =a(t)v_{1}(\mathbf{z},\overline{\mathbf{z}})(t)+b(t)v_{2}(\mathbf{z},\overline{\mathbf{z}})(t).\label{eq:5-22}
\end{align}
Let $I=\left\{ i\mid z_{i}(t)\not\not\equiv0\right\}$. Without loss
of generality we can assume $I=\left\{ 1,\ldots,m\right\} $, then
we have:
\begin{alignat*}{1}
z_{i}(t) & =a_{i}t^{p_{i}}+\mathrm{h.o.t.,\qquad\mathit{\mathrm{where\,}a_{j}\neq}0,\,}p_{i}\in\mathbb{Z},\, i\in I.\\
f(\mathbf{z}(t),\overline{\mathbf{z}}(t)) & =bt^{q}+\mathrm{h.o.t.},\qquad\mathrm{where}\, b\in\mathbb{C}^{*},\, q\in\mathbb{Z},\, q<0\\
a(t) & =\lambda_{0}t^{v_{0}}+\mathrm{h.o.t.},\qquad\mathrm{where}\,\lambda_{0}\in\mathbb{R},\, v_{0}\in\mathbb{Z}\\
b(t) & =\beta_{0}t^{v_{0}}+\mathrm{h.o.t.},\qquad\mathrm{where}\,\beta_{0}\in\mathbb{R},\, v_{0}\in\mathbb{Z}
\end{alignat*}
where $\left|\lambda_{0}\right|+\left|\beta_{0}\right|\neq0$. If
$\lambda_{0}\in\mathbb{R}^{*}$, then, by our assumption $a(t)<0$,
we have $\lambda_{0}<0$. To shorten notation, we write $\mathbf{a}=(a_{1},\ldots,a_{m})\in\mathbb{C}^{*I}$,
$\mathbf{P}=(p_{1},\ldots,p_{m})\in\mathbb{R}^{m}$ and consider the
linear function $l_{\mathbf{P}}=\sum_{i=1}^{m}p_{i}x_{i}$ defined
on $\overline{\mathrm{supp(\mathit{f}^{I})}}$. Let $\triangle$ be
the \emph{maximal face} of $\overline{\mathrm{supp(\mathit{f^{I}})}}$
where $l_{\mathbf{P}}$ takes its minimal value, say this value is
$d_{\mathbf{P}}$. We have $d_{\mathbf{P}}\leq\mathrm{ord_{t}}(f(\mathbf{z}(t),\overline{\mathbf{z}}(t))=q<0$.
By the above expansions, we get from \eqref{eq:5-22}:
\begin{alignat}{1}
\lambda_{0}(\frac{\frac{\overline{\partial f_{\triangle}^{I}}}{\partial z_{i}}(\mathbf{a},\overline{\mathbf{a}})}{\overline{b}}+\frac{\frac{\partial f_{\triangle}^{I}}{\partial\overline{z}_{i}}(\mathbf{a},\overline{\mathbf{a}})}{b})+i\beta_{0}(\frac{\frac{\overline{\partial f_{\triangle}^{I}}}{\partial z_{i}}(\mathbf{a},\overline{\mathbf{a}})}{\overline{b}}-\frac{\frac{\partial f_{\triangle}^{I}}{\partial\overline{z}_{i}}(\mathbf{a},\overline{\mathbf{a}})}{b}) & =
\begin{cases}
a_{i}, & \quad\mathrm{if}\, d_{\mathbf{P}}-p_{i}-q+v_{0}=p_{i}.\\
0, & \quad\mathrm{if}\, d_{\mathbf{P}}-p_{i}-q+v_{0}<p_{i}.
\end{cases}\label{eq:5-23}
\end{alignat}
Let $J=\left\{ j\in I\mid d_{\mathbf{P}}-p_{j}-q+v_{0}=p_{j}\right\}$.
We observe $J=\{j\in I\mid p_{j}=p=\underset{j\in I}{\min}\{p_{j}\}<0\}$.
If $J=\emptyset$, then from \eqref{eq:5-23}, we have $\mathbf{a}\in\Sing f_{\triangle}^{I}$.
Since $d_{\mathbf{P}}<0$, by \cite[Lemma 3.1]{CT}, we conclude that $\triangle$ is a face
of $\Gamma^{+}(f^{I})$. This contradicts the Newton strongly non degeneracy
of $f^{I}$. Hence $J\neq\emptyset$. To deduce the contradiction,
consider the following expansion:
\begin{align*}
\frac{\lambda_{0}+i\beta_{0}}{\overline{b}}\frac{\mathrm{d}\overline{f}(\mathbf{z}(t),\overline{\mathbf{z}}(t))}{\mathrm{d}t}+\frac{\lambda_{0}-i\beta_{0}}{b}\frac{\mathrm{d}f(\mathbf{z}(t),\overline{\mathbf{z}}(t))}{\mathrm{d}t}\\
=2\lambda_{0}qt^{q-1}+\mathrm{h.o.t.}
\end{align*}
We also have:\[
\frac{\mathrm{d}f(\mathbf{z}(t),\overline{\mathbf{z}}(t))}{\mathrm{d}t}=\left[\left\langle \mathbf{Pa},\overline{\mathrm{d}f_{\triangle}^{I}}(\mathbf{a},\overline{\mathbf{a}})\right\rangle +\left\langle \mathbf{P}\mathbf{\overline{\mathbf{a}}},\overline{\bar{\mathrm{d}}f_{\triangle}^{I}(\mathbf{a},\overline{\mathbf{a}})}\right\rangle \right]t^{d_{\mathbf{P}}-1}+\mathrm{h.o.t}.\]
By \eqref{eq:5-23}, we obtain:
\begin{align*}
\frac{\lambda_{0}+i\beta_{0}}{\overline{b}}\frac{\mathrm{d}\overline{f}(\mathbf{z}(t),\overline{\mathbf{z}}(t))}{\mathrm{d}t}+\frac{\lambda_{0}-i\beta_{0}}{b}\frac{\mathrm{d}f(\mathbf{z}(t),\overline{\mathbf{z}}(t))}{\mathrm{d}t}\\
=(2\sum_{j\in J}p\Vert a_{j}\Vert^{2})t^{d_{\mathbf{P}}-1}+\mathrm{h.o.t.}
\end{align*}
Since $d_{\mathbf{P}}\leq q$, comparing the two expansions of $\frac{\lambda_{0}+i\beta_{0}}{\overline{b}}\frac{\mathrm{d}\overline{f}(\mathbf{z}(t),\overline{\mathbf{z}}(t))}{\mathrm{d}t}+\frac{\lambda_{0}-i\beta_{0}}{b}\frac{\mathrm{d}f(\mathbf{z}(t),\overline{\mathbf{z}}(t))}{\mathrm{d}t}$,
It follows that $d_{\mathbf{P}}=q$ and $\lambda_{0}=\frac{\sum_{j\in J}p\Vert a_{j}\Vert^{2}}{q}>0$
from $p<0$ and $q<0$. This contradicts $\lambda_{0}<0$.
\end{proof}
\begin{remark}
In the holomorphic setting, the parallel results of this proposition
are \cite[Lemma 4.4]{Mi} and \cite[Lemma 4 and Lemma 5]{NZ2}; in the mixed setting,
this is a global analogue of \cite[Lemma 34]{Oka2}.
\end{remark}
\paragraph*{\bf{Proof of Theorem \ref{thm:M fib}}}
The proof is done as in the case of  a holomorphic polynomial. The strong non-degeneracy of $f$ yields a global fibration:
\[
f_{\mid}:f^{-1}(S_{\delta}^{1})\rightarrow S_{\delta}^{1}\]
where $\delta>0$ is sufficiently large. Since $S_{\delta}^{1}$ is compact
and $f(\Sing f)\cup S(f)$ is bounded, there exists $R_{0}>0$ sufficiently
large such that all the fibers intersect $S_{R}$ transversely for
any $R\geq R_{0}$. We therefore get the restriction\[
f_{\mid}:f^{-1}(S_{\delta}^{1})\cap B_{R}\rightarrow S_{\delta}^{1}\]
which is equivalent to the global fibration. By Proposition \ref{pro:Under-the-same},
there exists a non-zero vector field $\omega$ on $N=\{\mathbf{z}\in B_{R}\mid\left|f(\mathbf{z},\overline{\mathbf{z}})\right|\geq\delta\}$
such that\[
\begin{cases}
\mathrm{Re}\left\langle w(\mathbf{z}),v_{2}(\mathbf{z},\overline{\mathbf{z}})\right\rangle  & =0\\
\mathrm{Re}\left\langle w(\mathbf{z}),v_{1}(\mathbf{z},\overline{\mathbf{z}})\right\rangle  & >0\\
\mathrm{Re}\left\langle w(\mathbf{z}),\mathbf{z}\right\rangle >0.
\end{cases}\]
Along the integral curve $\gamma(t,\mathbf{z}_{0})$ of $w$ with
$\gamma(0,\mathbf{z}_{0})=\mathbf{z}_{0}\in N$, it is easily seen that the argument of $f(\gamma(t,\mathbf{z}_{0}),\overline{\gamma(t,\mathbf{z}_{0})})$
is constant and $\left|f(\gamma(t,\mathbf{z}_{0}),\overline{\gamma(t,\mathbf{z}_{0})})\right|$,
$\left\Vert \gamma(t,\mathbf{z}_{0})\right\Vert $ are monotone increasing.
Thus for every $\mathbf{z}_{0}\in N$, there exists a unique $h(\mathbf{z}_{0})\in S_{R}^{2n-1}\setminus f^{-1}(D_{\delta})$
and $t_{0}\in\mathbb{R}_{+}$ such that $\left\Vert \gamma(t_{0},h(\mathbf{z}_{0}))\right\Vert =R$.
Consequently, there is an isomorphism $\phi:f^{-1}(S_{\delta}^{1})\cap B_{R}\rightarrow S_{R}^{2n-1}\setminus f^{-1}(D_{\delta})$.
We therefore get $\frac{f}{\left|f\right|}:S_{R}^{2n-1}\setminus f^{-1}(D_{\delta})\longrightarrow S^{1}$
a locally trivial fibration which is equivalent to the fibration
$f_{\mid}:f^{-1}(S_{\delta}^{1})\cap B_{R}\rightarrow S_{\delta}^{1}$.
So $\frac{f}{\left|f\right|}:S_{R}^{2n-1}\setminus f^{-1}(D_{\delta})\longrightarrow S^{1}$
is also equivalent to the global one. This completes our proof.
\fin\\
\paragraph*{\bf{Proof of Corollary \ref{cor:mixed global}}}
From Remark \ref{rem:conve}, it follows that $S(\varphi)=\emptyset$
and $M(\varphi)$ is bounded. Thus we have $\frac{f}{\left|f\right|}:S_{R}^{2n-1}\setminus K\longrightarrow S^{1}$
is a locally trivial fibration. Note that the proof of Theorem \ref{thm:M fib}
yields that this fibration is equivalent to the global fibration:\[
f_{\mid}:f^{-1}(S_{\delta}^{1})\rightarrow S_{\delta}^{1}\]
where $\delta>0$ is sufficient large.
\fin\\
\begin{example}
\cite[Example 5 IV]{Oka2}\label{ex:semitame}
Consider a mixed polynomial\[
f(\mathbf{z},\overline{\mathbf{z}})=\frac{1}{4}z_{1}^{2}-\frac{1}{4}\overline{z}_{1}^{2}+z_{1}\overline{z}_{1}-(1+i)(z_{1}+z_{2})(\overline{z}_{1}+\overline{z}_{2}).\]
Then we have:
 \end{example}
\begin{enumerate}
\item $f$ is not Newton strongly non-degenerate at infinity and $S(f)=\emptyset$.
\item $\Sing f=\{\mathbf{z}\in\mathbb{C}^{2}\mid z_{1}=0,z_{2}\in\mathbb{C}\}\cup\{\mathbf{z}\in\mathbb{C}^{2}\mid z_{1}+z_{2}=0,z_{1}-i\overline{z}_{1}=0\}\cup\{\mathbf{z}\in\mathbb{C}^{2}\mid z_{1}+z_{2}=0,z_{1}+i\overline{z}_{1}=0\}$
\item $M(\varphi)$ is not bounded and $S(\varphi)=\{-\frac{1+i}{\sqrt{2}},\frac{2\pm i}{\sqrt{5}}\}$.
\end{enumerate}
\begin{remark}\label{rem:mixed semitame}
The above example is due to Oka. In the holomorphic case, N\'emethi
and Zaharia proved the existence of the Milnor fibration
at infinity for semitame polynomials in \cite{NZ2}. The definition
of semitame is equivalent to $S(f)\subset\{0\}$. But this example
shows that in the mixed case, the condition $S(f)\subset\{0\}$ fails
to insure the existence of the Milnor fibration $\frac{f}{\left|f\right|}$ at infinity. We also
observe that the Newton strong non-degeneracy condition at infinity of Theorem \ref{thm:M fib} can not be replaced by Newton non-degeneracy condition at infinity.
\end{remark}
\section*{Acknowledgement}
The author wish to thank Professor Mihai Tib\u{a}r for his help and support during the preparation for this paper which is also a work contained in the thesis research \cite{Ch}.

\end{document}